\documentclass[]{interact}
\usepackage{graphicx}
\usepackage{amsmath}
\usepackage{lineno,hyperref}
\usepackage{geometry}
\usepackage{array}
\usepackage{amsfonts}
\usepackage{amssymb}
\usepackage{setspace}
\usepackage{epstopdf}% To incorporate .eps illustrations using PDFLaTeX, etc.
\usepackage[caption=false]{subfig}% Support for small, `sub' figures and tables
\usepackage{natbib}% Citation support using natbib.sty
\RequirePackage{fix-cm}

\bibpunct[, ]{(}{)}{;}{a}{}{,}% Citation support using natbib.sty
\theoremstyle{plain}% Theorem-like structures provided by amsthm.sty

\newtheorem{thm}{Theorem}

\newtheorem{lem}{Lemma}

\makeatletter
\newcommand{\thickhline}{%
    \noalign {\ifnum 0=`}\fi \hrule height 1pt
    \futurelet \reserved@a \@xhline
}

\providecommand\implies{\DOTSB\;\Longrightarrow\;}
\providecommand\impliedby{\DOTSB\;\Longleftarrow\;}

\newcolumntype{"}{@{\hskip\tabcolsep\vrule width 1pt\hskip\tabcolsep}}

\newcommand{\RR}{\emph{New1}}
\newcommand{\BVV}{\emph{New2}}

\begin{document}

\articletype{}% Specify the article type or omit as appropriate

\title{Solving the signed Roman domination and signed total Roman domination problems with exact and heuristic methods}

\author{
\name{V. Filipovi\'c\textsuperscript{a}, D. Mati\'c\textsuperscript{b} and A. Kartelj\textsuperscript{a}\thanks{D. Mati\'c,  email: dragan.matic@pmf.unibl.org}}
\affil{\textsuperscript{a}University of Belgrade, Faculty of Mathematics, Studentski trg 16, Belgrade, Serbia \textsuperscript{b}University of Banjaluka, Faculty of and Natural Science and Mathematics, Mladena Stojanovi\'ca 2, Banja Luka, Bosnia and Herzegovina}
}

\maketitle

\begin{abstract}
In this paper we deal with the signed Roman domination and signed total Roman domination problems. For each problem we propose two integer linear programming (ILP) formulations, the constraint programming (CP) formulation and variable neighborhood search (VNS) method. We present proofs for the correctness of the ILP formulations and a  polyhedral study in which we show that the polyhedrons of the two model relaxations are equivalent.

VNS uses specifically designed penalty function that allows the appearance of slightly infeasible solutions. The acceptance of these solutions directs the overall search process to the promising areas in the long run.

All proposed approaches are tested on the large number of instances. Experimental results indicate that all of them reach optimal solutions for the most of small and middle scale instances. Both ILP models have proven to be more successful than the other two methods.
\end{abstract}


\section{Introduction}
Roman domination has been intensively investigated in recent years. Scientific interest in this problem arose after the papers "Defendens imperium romanum: a classical problem in military strategy" \citep{revelle2000defendens} and "Defend the Roman Emprire!" \citep{stewart1999defend} were published in two well known American journals. In order to protect the Roman territory from enemy attacks in the period of reduced power of the empire in the 4th century, the Emperor of Rome, Constantine the Great, issued a decree which defines the strategy of locating and moving legions through the regions in order to increase  security of the Empire. A region is considered unsecured if no legions are stationed, while regions are secured if  there is at least one legion there. Legions can move from a secure to an adjacent insecure region only if the old region is still secured, i.e. if one legion is ordered to move to the other region, then another legion still remains there. The Emperor's intention was to organize his legions in such a way that as few legions as
possible are stationed, but they still can defend the Empire according to his strategy.

Motivated by the problem that the glorious Roman emperor Constantine the Great faced in the 4th century, \citep{cockayne2004roman} formally introduced the Roman domination problem (RDP) as a mathematical problem and opened the wide area for further research. There are several variants of RDP, and they are mostly related to the conditions in which the vertices or edges are dominated, or to the introduction of some extra properties to the starting problem.

 Among many invariants of the problem, we notice some of them:
 independent Roman domination \citep{targhi2012properties,chellali2012note},
   weak Roman domination \citep{henning2003defending,cockayne2003secure},
   total Roman domination \citep{ahangar2016total},
  edge Roman domination \citep{chang2014edge},
   signed Roman edge domination \citep{ahangar2016signed},
   Roman k-domination \citep{hansberg2009upper,kammerling2009roman,bouchou2014relations},
  twin signed Roman domination \citep{sheikholeslami2016twin}
signed Roman k-domination \citep{henning2015signed}
Roman k-tuple domination number of a graph \citep{kazemi2014roman},
distance Roman domination \citep{aram2013distance},
strong Roman domination with multiple attacks \citep{henning2003defending,alvarez2015strong},
mixed Roman Domination \citep{ahangar2015mixed},
restrained Roman domination in graphs \citep{pushpam2015restrained},
 signed qRoman k-domination number \citep{adanza2015signed} and
double Roman domination \citep{beeler2016double}.

% signed Roman domination [4, 23], ovo je nas, 4 je ahangar, 23 je digraph
Domination in graphs is a developing research area with various applications to different kinds of networks, like social networks, biological networks, distributed networks etc. \citep{behtoei2014signed}. For example, some facility location problems can be modelled by Roman domination \citep{chambers2009extremal}. Dominating function can be interpreted as a cost function: if a vertex is assigned a number 1, that indicates that the vertex is chosen as a low-range hub which can serve only itself, if 2 is assigned to the vertex, that means that the vertex is a high-range hub that can serve neighboring locations, while the vertices assigned  a number less than 1 are not hubs. Distribution of the cheaper low-range hubs and more expensive high-range hubs in a wide network can be now considered as a Roman domination, where the task is to minimize the total number of the used hubs, with respect to the requirements that any non-hub node is covered by a high-range hub.

In this paper we deal with the signed Roman domination problem (SRDP) introduced in \citep{ahangar2014signed} and its sub-variant, signed total Roman domination problem(STRDP), introduced in \citep{volkmann2015signed}.
In the SRDP, there is a dominating function $f:V\rightarrow\{-1,1,2\}$ that satisfies two conditions:
for each vertex the sum of the values assigned to a vertex and its neighbors is at least 1 and for every vertex $i$ for which $f(i)=-1$ holds that it is adjacent to at least one vertex $j$ for which $f(j)=2$.
In the STRDP, the dominating function $f:V\rightarrow\{-1,1,2\}$  has to satisfy slightly different first condition: for each vertex the sum of the values assigned exactly to its neighbors is at least 1. The second condition is the same, i.e. for every vertex $i$ for which $f(i)=-1$ there must exists at least one vertex $j$ adjacent to $i$, satisfying $f(j)=2$.

%These two variants of RDP indicate more effective distribution of Roman legions in the neighbor cities of the Roman Empire in a special condition.
%The justification of this approach is follows. While in the origin RDP where dominating function enforces that each city is either defended by its own legion, or at least, there are enough legions in its neighborhood, in SRDP and STRDP a concept of "weak points" is introduced: regions with a negative (-1) number of legions, are considered as weak points, so in their neighborhood should be at least one region with 2 legions stationed. The goal of SRDP and STRDP  is again to find the optimal distribution of legions (the one with a minimal cardinality) that is \emph{able to preserve the Roman Empire safe}.

The signed Roman domination number (respectively signed total Roman domination number) is then defined as $\min \sum_{i \in V}{f(i)}$ over all functions satisfying the stated conditions for SRDP, respectively STRDP.

%The objective of the optimization problem related to any of these to problems is to identify the function $f$ satisfying the stated conditions, such that the  sum $\sum_{i \in V}{f(i)}$ is minimal.

% posle matematicke definicije mali primer sa eventualnom slikom (dati i optimalno resenje obavezno, bitno je da resenja za srtdp i srdp budu razlicita), ovo je onaj mali primer koji se kasnije pojavljuje u ILP formulacijama
% moze se dati komentar rezultata optimalnog resenja i za Roman domination i poredjenje tog
% resenja sa dobijenim resenjima za sr(t)dp

\section{Previous work}
In this section we briefly review the previous results on SRDP and STRDP and position our work in the literature.

Signed Roman domination problem was introduced in \citep{ahangar2014signed}. The authors presented several lower and upper bounds on the signed Roman domination number of arbitrary graphs, and determined the exact value of the signed Roman domination number for some special classes of graphs, like complete graphs, cycles and paths.
Signed Roman $k$-domination problem (SRkDP), introduced in \citep{henning2016signedk}, generalizes the SRDP. In SRkDP, for each vertex the sum of the values assigned to that vertex and its neighbors is at least $k$ and obviously, SRDP is a special case of SRkDP if $k=1$.

 Both SRDP and its generalization SRkDP have been  studied in literature from theoretical point of view.
In \citep{behtoei2014signed}, the signed Roman domination number of join graphs is analyzed. Beside some results for general join graphs, signed Roman domination number is determined for join of cycles, wheels, fans and friendship graphs. The signed Roman domination number for generalized Petersen graphs $P(n, k)$  is determined for   $k =  1$ and  $k=3$  in \citep{shirkolsigned2014}. SRDP is also considered in digraphs. In \citep{sheikholeslami2014signed} the authors presented several bounds, also determining the exact values for some special classes of digraphs.

 \cite{henning2015signed} and \cite{shao2017signed} present several theoretical results for SRkDP. In \citep{henning2015signed}, a tight lower bound on the SR2DP of a tree is determined. In \citep{shao2017signed}, it is proven that SRkDP is NP-complete even when restricted to bipartite and planar graphs. In the same paper, the exact values for thin torus graphs $C_3 \Box C_n$ and $C_4 \Box C_n$ are determined.

Similarly as for SRDP, results  in the literature for STRDP are focused on theoretical contributions.

Following the ideas in \citep{ahangar2014signed}, Volkmann  initiated the study of signed total Roman dominating number in \citep{volkmann2015signed}. The author found different bounds for signed total Roman dominating number and also determined the signed total Roman domination number of
some classes of graphs. In a consecutive work \citep{volkmann2016signed}, the same author continued for study STRDP and analyzed the signed total Roman domination number for some classes of graphs.
In a recent paper, \citet{zec2021signed} presented exact values and (lower and upper) bounds for signed (total) Roman domination number for several classes of planar graphs.

As mentioned above, SRDP and SRkDP have been studied only theoretically, by proposing results on  computational complexity and lower or upper bounds  either for  graphs in general or for certain graph classes. On the other side, in literature one cannot find any method for computational solving the mentioned problems.

 In our paper we mainly investigate SRDP and STRDP from the point of integer linear programming and variable neighborhood search, as methods for calculating the exact value (or at least upper bound)  of signed (total) roman domination number for an arbitrary graph.

\subsection{ILP formulations for related problems}
\label{subsection:ilp}
ILP models have been proposed for many variants of Roman domination problems. The first ILP formulation for the basic  Roman domination problem can be found  in an early work \citep{revelle2000defendens}. It was derived from the well-known problem, the Location Set Covering Problem \citep{toregas1971location}. \cite{revelle2000defendens} introduced the model two types of binary variables (usually denoted as \textbf{x} and \textbf{y}), indicating whether the corresponding vertices are assigned at least one (variables \textbf{x}) or two (variables \textbf{y}) legions or not. Another idea, presented, for example, in  \citep{burger2013binary}, also includes  two types of binary variables (also denoted as \textbf{x} and \textbf{y}), where the variables \textbf{x} (respectively \textbf{y}) are used to indicate if exactly one (respectively two) legion(s) is (are) placed on the corresponding region.   These two  ideas of using described two types of binary variables are used to construct new or to improve existing ILP models in a number of subsequent papers. \cite{burger2013binary} presented ILP formulations for solving three static domination problems: determining domination number, total domination number and Roman domination number, as well as two dynamic domination problems which include moving the guards from one region to another: weak Roman domination (WRD) and secure domination.
In dynamic domination problems, a set of additional binary variables were used to indicate a swap of guards between regions. \cite{ivanovic2016improved} improved the Revelle's and Burger's models for Roman domination by relaxing a set of binary variables to real. In addition, Ivanovi\'c proved that a set of constraints in Burger's model can be omitted.

 \cite{ivanovic2018improved}  made a contribution on solving WRD, by proposing the new ILP model, keeping the same sets of \textbf{x} and \textbf{y} binary variables, but reducing the total number of additional variables and  constraints.

Depending on the variant of a particular domination problem, new constraints and variable sets were introduced in ILP models.
\cite{ma2019integer} proposed three ILP models for solving the weighted total domination problem (WTDP). In all three models, binary variables (variables \textbf{x}) were used to indicate if a vertex is chosen for the solution or not. Remaining sets of variables, together with the appropriate constraints were used to control correctly measure the objective function. Following this approach, \cite{alvarez2019exact} proposed two alternative ILP formulations for WTDP, that outperform the results presented in \citep{ma2019integer}.

\cite{cai2019integer} propose several ILP models for a stronger version of RDP - double Roman domination problem DRDP, where  at most three legions stationed at each region are allowed. In DRDP, three types of  vertices with assigned legions are present, so an additional set of binary variables can be introduced, indicating whether a region is assigned three regions or not. In \citep{cai2019integer}, different models for DRDP follow the same idea, apart from the definition of some binary variables, also using some theoretical observation to reduce the total number of variables and constraints.

Following the idea from \citep{revelle2000defendens} and \citep{burger2013binary}, in this paper ILP models are formulated for SRDP and STRDP, as it is described in Section \ref{ILPformulations}. To our knowledge, SRDP and STRDP have not been modelled previously by linear programming and in this paper we propose the first ILP models for solving the mentioned problems.

\section{Problem definitions}
\label{problemdefinitions}

In this section formal mathematical definitions of SRDP and STRDP are presented. Since the problems are similar, two formulations share the same foundation,
differing in one set of constraints.

We first introduce the basic notation. Let $G = (V,E)$ represent a finite and undirected graph with the vertex set $V$ and the set of edges denoted by $E$.
The $open$ neighborhood of vertex $i$ is $N(i)=\{j \in V | ij \in E\}$, while
the $closed$ neighborhood is defined as $N[i] = N(i) \cup \{i\}$.
Let us further introduce a function $f$ on the set of vertices:

\begin{equation}
f: V \rightarrow \{-1,1,2\}
\label{eq:f}
\end{equation}

In both SRD and STRD problems the modified Roman domination function $f$ has to satisfy the condition that for every vertex $i \in V$, such
that $f(i) = -1$, there must be a vertex $j \in N(i)$ for which $f(j)=2$. Formally:

\begin{equation}
f(i)=-1 \implies ((\exists j)\;j \in N(i) \land f(j)=2), \quad i \in V
\label{eq:c1}
\end{equation}

Additional set of constraints for STRDP is given by:

\begin{equation}
s^{tot}(i) = \sum_{j \in N(i)}{f(j)} \geq 1,\quad i \in V,
\label{eq:c2strdp}
\end{equation}

and for SRDP:

\begin{equation}
s(i)=\sum_{j \in N[i]}{f(j)} \geq 1,\quad i \in V.
\label{eq:c2srdp}
\end{equation}

Finally, the objective function, also called the weight function on a graph $G=(V,E)$ is defined as:

\begin{equation}
\omega(f) = \sum_{i \in V}{f(i)}
\label{eq:w}
\end{equation}

The signed total Roman domination number and the signed Roman domination number
correspond to the minimal values of $\omega(f)$ for the graph $G$ in STRDP and SRDP, respectively.

\subsection{Example}

Let us consider an example graph given in Figure~\ref{figExampleGraphBlank}.
This graph consists of 6 nodes denoted by letters A, B, C, D, E and F
, and 9 edges: $\{A, B\},\;\{A, F\},\; \{B, C\},\; \{B, E\},\; \{C, D\},\; \{C, E\},\; \{D, E\}$ and $\{D, F\}$.

In the case of SRDP an optimal assignment is: $f(A)=-1,\;f(B)=2,\; f(C)=-1,\; f(D)=-1,\; f(E)=1$ and $f(F)=2$ and the signed Roman domination number is equal to 2.
The first set of constraints (\ref{eq:c1}) is satisfied since all nodes that are assigned the value of -1 have a neighbor that has value of 2, i.e. nodes A and C have neighbor B, while node D is neighboring F.
The second set of constraints (\ref{eq:c2srdp}) is also satisfied, since the sum of assigned values in the closed neighborhoods of all nodes is greater or equal to 1, i.e. $s(A)=3,\; s(B)=3,\; s(C)=1,\; s(D)=1,\; s(E)=1$ and $s(F)=2$.

In the case of STRDP an optimal assignment is different: $f(A)=-1,\; f(B)=1, f(C)=-1,\; f(D)=2,\; f(E)=1,\;$ and $f(F)=2$. The total signed Roman domination number is equal to 4.
Constraints (\ref{eq:c1}) are satisfied because node A is a neighbor with F and node C is a neighbor with D.
Finally, the sum of assigned values in the open neighborhood of any node is strictly higher than zero: $s^{tot}(A)=3,\; s^{tot}(B)=1,\; s^{tot}(C)=4,\; s^{tot}(D)=2,\; s^{tot}(E)=2$ and $s^{tot}(F)=2$.

 \begin{figure}[h!]
        \centering
           \subfloat[Graph setting]{%
              \includegraphics{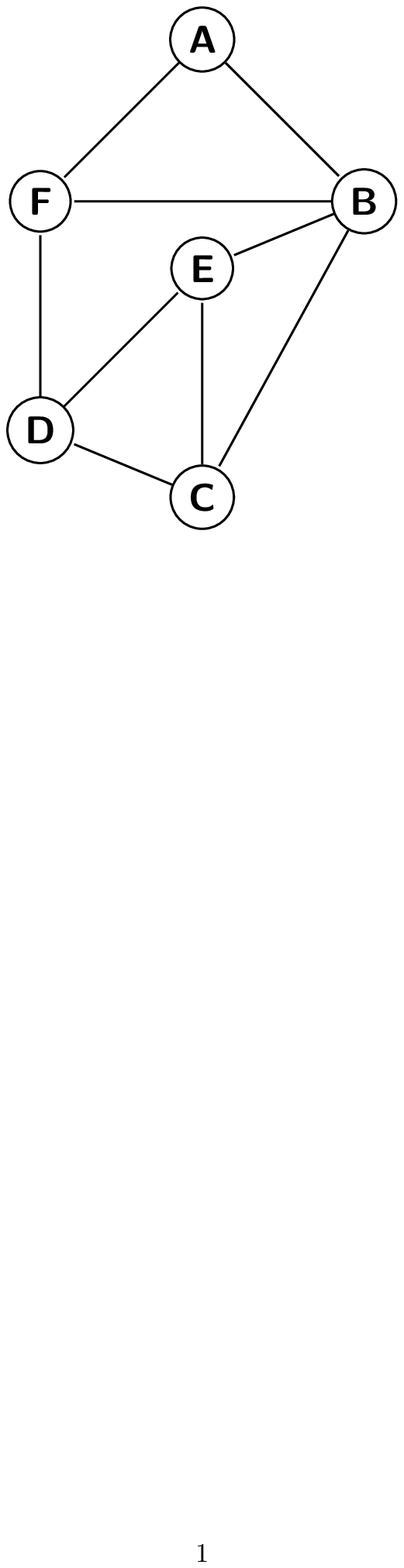}%
							\label{figExampleGraphBlank}
           }
           \subfloat[Optimal SRDP solution]{%
              \includegraphics{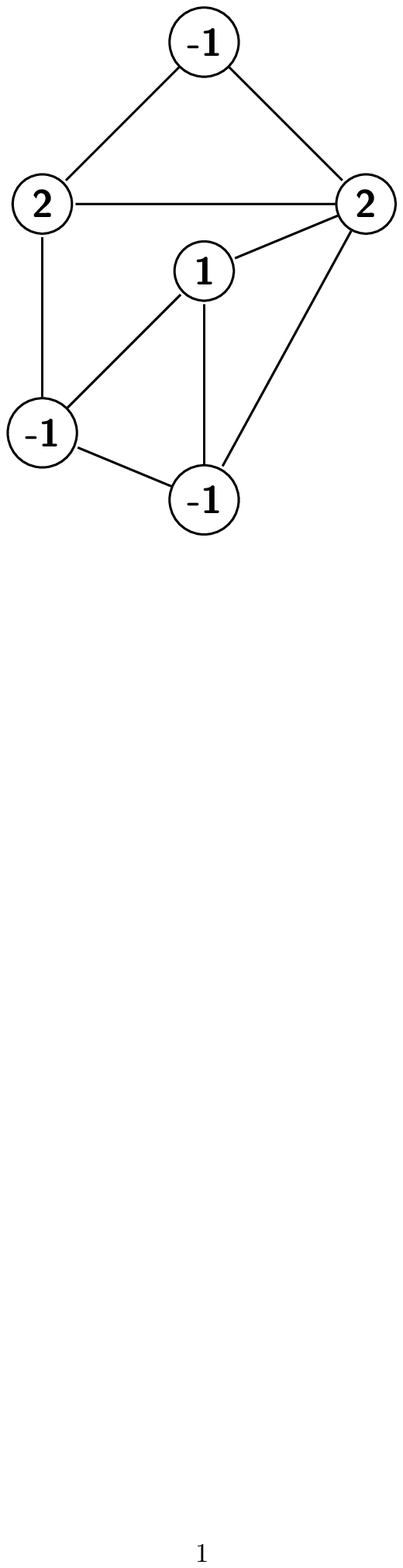}%
							\label{figExampleGraphSRDP}
           }
           \subfloat[Optimal STRDP solution]{%
              \includegraphics{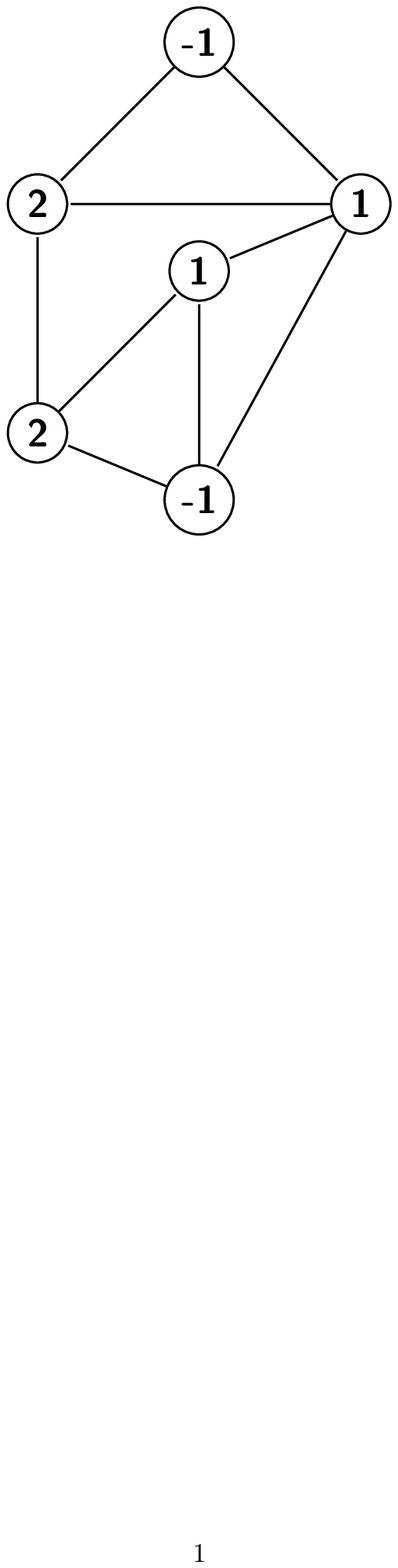}%
							\label{figExampleGraphSTRDP}
           }
           \caption{Example graph.}
           \label{figExampleGraph}
    \end{figure}

\section{Two ILP formulations}
\label{ILPformulations}
% 1. definicija konstanti odnosno parametara
% 2. definicija promenljivih
% 3. model
% 4. kratak opis ogranicenja, cemu sta sluzi
% 5. za svaki od modela teorema da je model ekvivalentan definiciji problema
% 6. rezultat malog primera iz cplex-a

In this section, two new integer linear programming formulations for each of two considered problems are proposed.
We denote these two ILP models as  $\RR$   and $\BVV$ for SRDP, while
the ILP models for STRDP take labels $\RR^{tot}$ and $\BVV^{tot}$.

\subsection{\emph{\RR} formulations}

For the function $f$ introduced in (\ref{eq:f}) and vertices $i \in V$ let us define the following variables:

\begin{equation}
\label{eq:RRx}
x_i=
\begin{cases}
1, &f(i) \geq 1\\
0, &otherwise.
\end{cases}
\end{equation}

\begin{equation}
y_i=
\begin{cases}
1, &f(i) = 2\\
0, &otherwise.
\end{cases}
\label{eq:RRy}
\end{equation}

The $\RR^{tot}$ formulation is then given as follows:

\begin{equation}
\min \sum_{i \in V} {\left(2 x_i + y_i -1\right)}
\label{eq:rrSTRDPf}
\end{equation}

\begin{equation}
x_i \geq y_i, \quad i \in V
\label{eq:rrSTRDP1}
\end{equation}

\begin{equation}
x_i + \sum_{j \in N(i)}{y_j} \geq 1,\quad i \in V
\label{eq:rrSTRDP2}
\end{equation}

\begin{equation}
\sum_{j \in N(i)}{\left(2 x_j + y_j -1\right)} \geq 1,\quad i \in V
\label{eq:rrSTRDP3}
\end{equation}

\begin{equation}
x_i, y_i \in \{0,1\}, \quad i \in V
\label{eq:rrSTRDP4}
\end{equation}

The objective function (\ref{eq:rrSTRDPf}) corresponds to the previously introduced weight function (\ref{eq:w}), i.e. total number of assigned legions.
Constraints (\ref{eq:rrSTRDP1}) and (\ref{eq:rrSTRDP4}) ensure that values of $f$ function are from $\{-1,1,2\}$, thus it discards the invalid combination $(x_i,y_i) = (0,1)$. Constraints (\ref{eq:rrSTRDP2}) forbid the situation in which $f(i)=-1$ and exist no $j \in V(i)$ such that $f(j)=2$, meaning there cannot exist a weakly defended city without strongly defended neighbor city. Finally, constraints (\ref{eq:rrSTRDP3}) enforce that all cities are surrounded with cities whose total number of legions is at least 1 (including special case cities with negative number of legions).

The $\RR$ formulation, related to SRDP is similar to $\RR^{tot}$ since it has the same objective function (\ref{eq:rrSTRDPf}) and the same constraints (\ref{eq:rrSTRDP1}), (\ref{eq:rrSTRDP2}) and (\ref{eq:rrSTRDP4}).
The difference is given by replacing (\ref{eq:rrSTRDP3}) with the following
set of constraints:

\begin{equation}
\sum_{j \in N[i]}{\left(2 x_j + y_j -1\right)} \geq 1,\quad i \in V
\label{eq:rrSRDP3}
\end{equation}

which, as in $\RR^{tot}$, restricts the total number of surrounding legions, but, here, the cities own legion(s) are included in the count.

\subsection{\emph{\BVV} formulations}

For a function $f$ and vertices $i \in V$ let us define the following variables:

\begin{equation}
x_i=
\begin{cases}
1, &f(i) = 1\\
0, &otherwise.
\end{cases}
\label{eq:BVVx}
\end{equation}

\begin{equation}
y_i=
\begin{cases}
1, &f(i) = 2\\
0, &otherwise.
\end{cases}
\label{eq:BVVy}
\end{equation}

The $\BVV^{tot}$ formulation is then given as follows:

\begin{equation}
\min \sum_{i \in V} {\left(2 x_i + 3 y_i -1\right)}
\label{eq:bvvSTRDPf}
\end{equation}

\begin{equation}
x_i + y_i\leq 1, \quad i \in V
\label{eq:bvvSTRDP1}
\end{equation}

\begin{equation}
x_i + y_i +\sum_{j \in N(i)}{y_j} \geq 1,\quad i \in V
\label{eq:bvvSTRDP2}
\end{equation}

\begin{equation}
\sum_{j \in N(i)}{\left(2 x_j + 3 y_j -1\right)} \geq 1,\quad i \in V
\label{eq:bvvSTRDP3}
\end{equation}

\begin{equation}
x_i, y_i \in \{0,1\}, \quad i \in V
\label{eq:bvvSTRDP4}
\end{equation}

The objective function (\ref{eq:bvvSTRDPf}) corresponds to previously introduced weight function (\ref{eq:w}), i.e. total number of assigned legions.
Constraints (\ref{eq:bvvSTRDP1}) and (\ref{eq:bvvSTRDP4}) ensure that values of $f$ function are from $\{-1,1,2\}$, thus it discards the invalid combination $(x_i,y_i) = (1,1)$. Constraints (\ref{eq:bvvSTRDP2}) have the same significance as (\ref{eq:rrSTRDP2}) in $\RR^{tot}$ and $\RR$ models, i.e. they disallow legion constellation in which $f(i)=-1$ for some $i \in V$ and exist no $j \in V(i)$ such that $f(j)=2$. Constraints (\ref{eq:bvvSTRDP3}) enforce total number of surrounding legions to be at least 1 for each city.

$\BVV$ follows the same idea as $\BVV^{tot}$, so the objective function (\ref{eq:bvvSTRDPf}) and constraints (\ref{eq:bvvSTRDP1}), (\ref{eq:bvvSTRDP2}), (\ref{eq:bvvSTRDP4}) are the same. The only difference is in the usage of $closed$ neighborhood in SRDP.

\begin{equation}
\sum_{j \in N[i]}{\left(2 x_j + 3 y_j -1\right)} \geq 1,\quad i \in V
\label{eq:bvvSRDP3}
\end{equation}

\subsection{Equivalence proofs}
In this subsection we present the theorems which guarantee the equivalence between mathematical and two ILP formulations for each of two problems.

\begin{thm}
Optimal objective function value of the \emph{$\RR^{tot}$} formulation (\ref{eq:rrSTRDPf})-(\ref{eq:rrSTRDP3}) is equal to the optimal objective function value of the formulation given by (\ref{eq:f}), (\ref{eq:c1}) and (\ref{eq:c2strdp}).
\end{thm}
\begin{proof}
($\implies$) Let a feasible solution of the  formulation (\ref{eq:rrSTRDPf})-(\ref{eq:rrSTRDP4}) be denoted by two $n$-dimensional vectors $\mathbf{x}=(x_1,...,x_n)$ and $\mathbf{y}=(y_1,...,y_n)$ where $n = |V|$.

Let $f$ be the function on the set $V$ defined as:

\begin{equation}\label{proof2:f}
f(i)=
\begin{cases}
-1,&x_i=y_i=0\\
1, &x_i=1, y_i=0\\
2, &x_i=1, y_i=1.
\end{cases}
\end{equation}

Because of the constraints (\ref{eq:rrSTRDP1}), the case  $(x_i,y_i)=(0,1)$ is not possible, so the function $f$ is well defined according to (\ref{eq:f}).

Let $i\in V$. To prove that the condition (\ref{eq:c2strdp}) is satisfied, we consider three possible cases:
\begin{enumerate}
\item $\underline{x_i=y_i=0}$: $f(i)=-1$ and $2 x_i+ y_i -1 = 2 * 0+ 0 -1 = -1;$
\item $\underline{x_i=1, y_i=0}$: $f(i)=1$ and $2 x_i+ y_i -1 = 2 * 1+ 0 -1 = 1;$
\item $\underline{x_i=1, y_i=1}$: $f(i)=2$ and $2 x_i+ y_i -1 = 2 * 1+ 1 -1 = 2.$
\end{enumerate}

In all three cases we have that for each $i\in V$, $f(i)=2 x_i+ y_i -1$. From the constraint (\ref{eq:rrSTRDP3}), we have

$1\leq \sum_{j\in N(i)}(2x_j+y_j-1)=\sum_{j\in N(i)}f(i)$ and (\ref{eq:c2strdp}) holds.

Let us prove (\ref{eq:c1}). Let $i \in V$ and $f(i)=-1$. From the definition (\ref{proof2:f}) of the function $f$, we have $x_i=0$. From (\ref{eq:rrSTRDP2}) and from the binary nature of the $\mathbf{y}$ variables ensured by (\ref{eq:rrSTRDP4}), we have:

$0+\sum_{j\in N(i)}y_j\geq 1\implies ((\exists j\in N(i))\ y_j=1)$.

Because of (\ref{eq:rrSTRDP1}) also $x_j=1$. From (\ref{proof2:f}), $f(j)=2$ and (\ref{eq:c1}) holds.

Let $f_{opt}$ be the optimal solution for the STRDP w.r.t. (\ref{eq:f}) - (\ref{eq:c2strdp}). Let $opt_{\RR^{tot}}(\mathbf{x},\mathbf{y})$ denote an optimal solution w.r.t.
$\RR^{tot}$ formulation (\ref{eq:rrSTRDPf})-(\ref{eq:rrSTRDP4}).

Then, $opt_{\RR^{tot}}(\mathbf{x},\mathbf{y}) = \sum_{i \in V}{2 \overline{x_i} +  \overline{y_i} -1} = \sum_{i \in V}{\overline{f(i)}}$ w.r.t. definitions (\ref{eq:RRx}) and (\ref{eq:RRy}). It is proven that $\overline{f}$ corresponds to one Roman domination function w.r.t. (\ref{eq:f})-(\ref{eq:c2srdp}).
Therefore the following holds: $opt_f \leq \sum_{i \in V}{\overline{f(i)}} = opt_{\RR^{tot}}(\mathbf{x},\mathbf{y})$.

($\impliedby$) Let $f$ be a function satisfying (\ref{eq:f})-(\ref{eq:c2srdp}). We define the variables $\mathbf{x}$ and $\mathbf{y}$ according to (\ref{eq:RRx}) and (\ref{eq:RRy}).

Binary nature of the variables $\mathbf{x}$ and $\mathbf{y}$, i.e. the constraints (\ref{eq:rrSTRDP4}) are directly implied from the definitions (\ref{eq:RRx}) and (\ref{eq:RRy}).

Also, from the definition of the variables $\mathbf{x}$ and $\mathbf{y}$, for each $i\in V$, the case $(x_i,y_i)=(0,1)$ is excluded, therefore the constraints (\ref{eq:rrSTRDP1}) are satisfied.

Further, for each $i\in V$ we have

\begin{enumerate}
\item $f(i)=-1 \implies x_i=y_i=0 \implies 2 x_i +  y_i -1 = -1$;
\item $f(i)=1 \implies x_i=1, y_i=0 \implies 2 x_i +  y_i -1 = 1$;
\item$f(i)=2 \implies x_i=1, y_i=1 \implies 2 x_i + y_i -1 = 2$.
\end{enumerate}

In all three cases we have that $f(i)=2 x_i +  y_i -1 $. Let us prove the constraints (\ref{eq:rrSTRDP3}).

Let $i\in V$. From (\ref{eq:c2strdp}) we have

$1\leq \sum_{j\in N(i)}f(j) = \sum_{j\in N(i)}(2 x_i +  y_i -1 )$.

Let us now prove (\ref{eq:rrSTRDP2}). Let $i\in V$. We analyze three cases:
\begin{enumerate}
\item $f(i) = 1 \implies x_i=1$ and (\ref{eq:rrSTRDP2}) is satisfied because $\mathbf{y}$ variables are non-negative.
\item $f(i) = 2 \implies x_i=1$ again and (\ref{eq:rrSTRDP2}) is satisfied because of the same reason
\item $f(i)= -1 \implies x_i=0$. Because of (\ref{eq:c1}) and the definition of $\mathbf{y}$ variables, we have

 $((\exists j\in N(i)) f(i)=2)\implies ((\exists j\in N(i))\ y_j=1)\implies x_i+\sum_{j\in N(i)}y_i\geq 1$ and (\ref{eq:rrSTRDP2}) is satisfied.

\end{enumerate}

Let $opt_f$ be the optimal solution of the STRD problem, i.e.  $f$ is the function for which  $\sum_{i\in V}f(i)$ is minimal.  Let $\overline{\mathbf{x}}$ and $\overline{\mathbf{y}}$ be the variables of the \emph{$\RR^{tot}$} model, defined according to (\ref{eq:RRx}) and (\ref{eq:RRy}). It is proven that $\overline{\mathbf{x}}$ and $\overline{\mathbf{y}}$ satisfy the constraints (\ref{eq:rrSTRDPf})-(\ref{eq:rrSTRDP4}) and therefore,

$obj_{\RR^{tot}}\leq \sum_{i\in V}(2\overline{x_i}+\overline{y_i}-1)=\sum_{i\in V}f(i)=opt_f$.

\end{proof}

%\begin{proof}
%dokazan je BVV, a ovaj bi trebalo slicno.
%\end{proof}

\begin{thm}
Optimal objective function value of the \emph{$\RR$} formulation (\ref{eq:rrSTRDPf}), (\ref{eq:rrSTRDP1}), (\ref{eq:rrSTRDP2}), (\ref{eq:rrSRDP3}) and (\ref{eq:rrSTRDP4}) is equal to the optimal objective function value of the formulation given by (\ref{eq:f}), (\ref{eq:c1}) and (\ref{eq:c2srdp}).
\end{thm}

%\begin{proof}
% ovaj je isti kao prethodni thmRRT, samo obratiti paznju da skup suseda cvora ovde ukljucuje i sam cvor.
%\end{proof}

The proof of Theorem 2 will be omitted since it is similar to the proof of Theorem~1.

\begin{thm}
Optimal objective function value of the \emph{$\BVV^{tot}$} formulation (\ref{eq:bvvSTRDPf})-(\ref{eq:bvvSTRDP4}) is equal to the optimal objective function value of the formulation given by (\ref{eq:f}), (\ref{eq:c1}) and (\ref{eq:c2strdp}).
\end{thm}

\begin{proof}
($\implies$) Let us denote a feasible solution to formulation (\ref{eq:bvvSTRDPf})-(\ref{eq:bvvSTRDP4}) by two $n$-dimensional vectors $\mathbf{x}=(x_1,...,x_n)$ and $\mathbf{y}=(y_1,...,y_n)$ where $n = |V|$ and let $f$ be a function on the set $V$. The following holds:

\begin{equation}
f(i)=
\begin{cases}
-1,&x_i=y_i=0\\
1, &x_i=1, y_i=0\\
2, &x_i=0, y_i=1.
\end{cases}
\end{equation}

The function $f$ is well defined function because the set of constraints (\ref{eq:bvvSTRDP1}) of $\BVV^{tot}$ disallows the remaining possibility $(x_i,y_i)=(1,1)$. Hence, the function $f$ is defined w.r.t. (\ref{eq:f}).

%The next step is to prove that $\sum_{j \in N(i)}{f(j)} \geq 1$ (\ref{eq:c1}) follows from
%$\sum_{j \in N}{(2 x_j + 3 y_j - 1 )} \geq 1$ (\ref{eq:bvvSTRDP3}) for $i\in V$.

Let us prove that $f$ satisfies the condition (\ref{eq:c2strdp}). Let $i \in V$ be a vertex. We first state the following:

\begin{enumerate}
\item $\underline{x_i=y_i=0}$: $f(i)=-1$ and $2 x_i+3 y_i -1 = 2 * 0+ 3*0 -1 = -1;$
\item $\underline{x_i=1, y_i=0}$: $f(i)=1$ and $2 x_i+3 y_i -1 = 2 * 1+ 3*0 -1 = 1;$
\item $\underline{x_i=0, y_i=1}$: $f(i)=2$ and $2 x_i+3 y_i -1 = 2 * 0+ 3*1 -1 = 2.$
\end{enumerate}

Since in all three cases $f(i)$ is equal to the expression $2 x_i+3 y_i - 1$, from (\ref{eq:bvvSTRDP3}) we get
$1\leq \sum_{j \in N(i)}{(2 x_j + 3 y_j - 1 )}= \sum_{j \in N(i)}f(i)$ and (\ref{eq:c2strdp}) is satisfied.

%under constraints $x_i, y_i \in \{0,1\}$ and $x_i+y_i \leq 1$, we conclude that (\ref{eq:c1}) holds.

Proving that (\ref{eq:c1}) holds also requires analysis by cases and consideration of constraints (\ref{eq:bvvSTRDP2}):
% $x_i+y_i+\sum_{j \in N(i)}y_j \geq 1, \; i \in V$ .

\begin{enumerate}
\item $\underline{x_i=1}$: $x_i + y_i\leq 1 \implies y_i=0 \implies f(i)=1$, so this case is irrelevant;
\item $\underline{y_i=1}$: $x_i + y_i\leq 1 \implies x_i=0 \implies f(i)=2$, so this case is irrelevant;
\item $\underline{x_i=1, y_i=1}$: this case is not feasible;
\item $\underline{x_i=0,\ y_i=0 \implies f(i)=-1}$. From (\ref{eq:bvvSTRDP2}) and because of the binary nature of $\mathbf{y}$ variables, we have

 $x_i+y_i+\sum_{j \in N(i)}y_j \geq 1 \implies (\exists j \in N(i))\ y_j=1 \implies (\exists j \in N(i)) f(j)=2$. Therefore  (\ref{eq:c1}) is satisfied.
\end{enumerate}

Let $opt_f$ denote optimal solution w.r.t. formulation (\ref{eq:f})-(\ref{eq:c2strdp}), and let $opt_{\BVV^{tot}}(\mathbf{x},\mathbf{y})$ denote an optimal solution w.r.t.
$\BVV^{tot}$ formulation (\ref{eq:bvvSTRDPf})-(\ref{eq:bvvSTRDP4}).
Then, $opt_{\BVV^{tot}}(\mathbf{x},\mathbf{y}) = \sum_{i \in V}{2 \overline{x_i} + 3 \overline{y_i} -1} = \sum_{i \in V}{\overline{f(i)}}$ w.r.t. definitions (\ref{eq:BVVx}) and (\ref{eq:BVVy}). Since $\bar{f}$ corresponds to one possible solution w.r.t. formulation (\ref{eq:f})-(\ref{eq:c2srdp}), it is at most good as $opt_f$, and therefore the following holds: $opt_f \leq \sum_{i \in V}{\overline{f(i)}} = opt_{\BVV^{tot}}(\mathbf{x},\mathbf{y})$.

($\impliedby$)

Let $f$ denote the function that satisfies (\ref{eq:f})-(\ref{eq:c2strdp}). Let us use definitions for the variables $\mathbf{x}$ and $\mathbf{y}$ given by (\ref{eq:BVVx}) and (\ref{eq:BVVy}). The following is true:

\begin{enumerate}
\item $f(i)=-1 \implies x_i=y_i=0 \implies 2 x_i + 3 y_i -1 = -1$;
\item $f(i)=1 \implies x_i=1, y_i=0 \implies 2 x_i + 3 y_i -1 = 1$;
\item$f(i)=2 \implies x_i=0, y_i=1 \implies 2 x_i + 3 y_i -1 = 2$.
\end{enumerate}

Afterwards, from (\ref{eq:c1}) it follows that $1 \leq \sum_{j \in N(i)}{f(j)} = \sum_{j \in N(i)}{(2 x_i + 3 y_i -1)}, \; i \in V$, hence, (\ref{eq:bvvSTRDP3}) is true.

Let us consider (\ref{eq:c2strdp}), i.e. $(f(i)=-1 \implies (\exists j)\;j \in N(i) \land f(j)=2), \quad i \in V$.
Three cases are of interest:

\begin{enumerate}
\item \underline{$f(i)=1$}: $x_i=1, y_i=0 \implies x_i+y_i+\sum_{j \in N(i)}{y_j} = 1+0+\sum_{j \in N(i)}{y_j} \geq 1$, because $\mathbf{x}$ and $\mathbf{y}$ variables cannot be negative;
\item \underline{$f(i)=2$}: $x_i=0, y_i=1 \implies x_i+y_i+\sum_{j \in N(i)}{y_j}=0+1+\sum_{j \in N(i)}{y_j} \geq 1$;
\item \underline{$f(i)=-1$}: from (\ref{eq:c1}) it follows $(\exists j) j \in N(i) \land f(j)=2 \implies y_j=1 \implies x_i+y_i+\sum_{j \in N(i)}{y_j}=0+0+\sum_{j \in N(i)}{y_j} \geq 1.$
\end{enumerate}

Therefore, (\ref{eq:bvvSTRDP2}) also holds.

Set of constraints (\ref{eq:bvvSTRDP1}) is also proven directly from variable definitions (\ref{eq:BVVx}) and (\ref{eq:BVVy}) and following considerations:

\begin{enumerate}
\item $f(i)=-1 \implies x_i=0, y_i=0 \implies x_i+y_i=0 \leq 1,\; i \in V$;
\item $f(i)=1 \implies x_i=1, y_i=0 \implies x_i+y_i=1 \leq 1,\; i \in V$;
\item $f(i)=2 \implies x_i=0, y_i=1 \implies x_i+y_i=1 \leq 1,\; i \in V$,
\end{enumerate}

Constraints that enforce variables to be binary (\ref{eq:bvvSTRDP4}) are directly implied by definitions (\ref{eq:BVVx}) and (\ref{eq:BVVy}).

Finally, if $opt_f$ denotes optimal solution w.r.t. (\ref{eq:f})-(\ref{eq:c2strdp}) it can be shown that: $opt_f = \sum_{i \in V}{\overline{f(i)}}= \sum_{i \in V}{(2 \overline{x_i} + 3 \overline{y_i} -1)} \geq  opt_{\BVV^{tot}}(\mathbf{x},\mathbf{y})$. This leads to the final conclusion that $ opt_{\BVV^{tot}}(\mathbf{x},\mathbf{y}) = opt_f$.
\end{proof}

\begin{thm}
Optimal objective function value of the \emph{$\BVV$} formulation (\ref{eq:bvvSTRDPf}), (\ref{eq:bvvSTRDP1}), (\ref{eq:bvvSTRDP2}), (\ref{eq:bvvSRDP3}) and (\ref{eq:bvvSTRDP4}) is equal to the optimal objective function value of the formulation given by (\ref{eq:f}), (\ref{eq:c1}) and (\ref{eq:c2srdp}).
\end{thm}

The proof of Theorem 4 is omitted since it is very similar to the proof of Theorem~3.

%\begin{proof}
% ovaj je isti kao prethodni thmRRT, samo obratiti paznju da skup suseda cvora ovde ukljucuje i sam cvor.
%\end{proof}

\section{Polyhedral analysis}

In this section we compare the two polyhedrons defined by two linear programming relaxations of $\BVV^{tot}$ and $\RR^{tot}$ formulations.

% RR^{R} treba preimenovati u npr. RR^{t} za total ili kako je vec gore u tekstu. Ovo R ne treba da opisuje relaksaciju
% jer LP vec to znaci.

The idea is based on the following consideration:
Let $LP_{\BVV^{tot}}$ and $LP_{\RR^{tot}}$ be the linear programming relaxations of two ILP models $\BVV^{tot}$ and $\RR^{tot}$ respectively, that are obtained by relaxing the integral (binary) variables from $\{0,1\}$ to the continuous variables from $[0,1]$.

Let us denote the variables of the $\BVV^{tot}$ formulation by $(x',y')$.% and variables of $RR^{R}$ formulation by $(x'',y'')$.

By introducing the following substitution
\begin{equation}\label{eqn:subs}
x''=x'+y'  \textnormal{ and } y''=y'.
\end{equation}
we project the solutions of the  $LP_{\BVV^{tot}}$ model into the space of solutions $(x'',y'')$ from $LP_{\RR^{tot}}$. By the reverse substitution, we will project the solutions of the $LP_{\RR^{tot}}$ to the space of the $LP_{\BVV^{tot}}$  model.

By using the established correspondence between the relaxed solutions, we will  prove that two LP relaxations describe the same polyhedron.

Also, consider the objective function of $\BVV^{tot}$ model (\ref{eq:bvvSTRDPf}):
$$\sum_{i \in V} {\left(2 x'_i + 3 y'_i -1\right)}=\sum_{i \in V} {\left(2 x'_i + 2 y'_i+y'_i -1\right)}=\sum_{i \in V} {\left(2 x''_i + y''_i -1\right)}$$
and it is equal to the objective function (\ref{eq:rrSTRDPf}) of the $\RR^{tot}$ model in terms of $x'',y''$ variables.

\begin{lem}\label{lemma1}
The polyhedron defined by $LP_{\BVV^{tot}}$ is contained in the polyhedron of $LP_{\RR^{tot}}$.
\end{lem}
\begin{proof}
Let $x',y'$ be a feasible solution of $LP_{\BVV^{tot}}$ and let $x'',y''$ be solutions defined by the substitution (\ref{eqn:subs}).

We first show that the variables $x''$ and $y''$ are well defined w.r.t. the relaxed $\RR^{tot}$ model.
Because of $x',y'\in[0,1]$ and (\ref{eq:bvvSTRDP1}) we have that $x'_i+y'_i=x''_i\in [0,1]$.
Since  $y''=y'$, it is obvious that $y''_i\in[0,1]$.

Since $x'_i\geq 0$ it holds $x'_i+y'_i\geq y'_i  \implies x''_i\geq y''_i$ and therefore the  inequality (\ref{eq:rrSTRDP1})  holds.

%The inequality (\ref{eq:rrSTRDP1}) for $x''$ and $y''$ holds from (\ref{eqn:subs}) and the fact that $y$ variables are non-negative.

The inequalities (\ref{eq:rrSTRDP2})  are obviously satisfied for $x''$ and $y''$, because of (\ref{eqn:subs}) and the constraints (\ref{eq:bvvSTRDP2}).

Let $i \in V$ and $j\in N(i)$. From
$$\sum_{j\in N(i)} {\left(2 x'_j + 3 y'_j -1\right)}=\sum_{j\in N(i)} {\left(2 x'_j + 2 y'_j+y'_j -1\right)}=\sum_{j\in N(i)} {\left(2 x''_j + y''_j -1\right)}$$
and  (\ref{eq:bvvSTRDP3}) we have
$\sum_{j\in N(i)} {\left(2 x''_j + y''_j -1\right)}\geq 1$, which proves that the inequality (\ref{eq:rrSTRDP3}) holds for $i \in V$.

All the constraints of the relaxed model $LP_{\RR^{tot}}$ are satisfied and the Lemma is proven.

\end{proof}

\begin{lem}\label{lemma2}
The polyhedron defined by $LP_{\RR^{tot}}$  is contained in the polyhedron of $LP_{\BVV^{tot}}$.
\end{lem}
\begin{proof}
Let $x'',y''$ be a feasible solution of $LP_{\RR^{tot}}$.
We introduce the reverse projection of the variables $x'',y''$  to the variables  $x',y'$ of a solution for LP1.

\begin{equation}\label{eqn:subs2}
x'=x''-y''  \textnormal{ and } y'=y''.
\end{equation}

Firstly, we prove that the relaxed variables $x'_i,y'_i\in [0,1]$. Since $x''_i\leq 1$ and $y''_i\geq 0 \implies -y''_i\leq 0$ we have $x'_i = x''_i- y''_i\leq 1$. Further, from  (\ref{eq:rrSTRDP1}) we have $x'_i= x''_i-y''_i\geq 0$. Therefore, $x'_i\in [0,1]$. Since $y'_i=y''_i$, it is clear that $y'_i\in [0,1]$. That shows $x'$ and $y'$ are well defined. Further, $x''_i\leq 1 \implies x''_i-y''_i+y''_i\leq 1 \implies x'_i+y'_i\leq 1$ and the inequality (\ref{eq:bvvSTRDP1}) holds.

The inequality (\ref{eq:bvvSTRDP2}) is directly satisfied because of (\ref{eqn:subs2}) and (\ref{eq:rrSTRDP2}). Let $i \in V$ and $j\in N(i)$.  From (\ref{eq:rrSTRDP3}) we have $\sum_{j\in N(i)} {\left(2 x''_j + y''_j -1\right)}\geq 1$ $\implies$ $\sum_{j\in N(i)} {\left(2 x''_j -2y''_j + 3y''_j -1\right)}\geq 1 $ $\implies$ $\sum_{j\in N(i)} {\left(2 x'_j  + 3y'_j -1\right)}\geq 1 $, which proves (\ref{eq:bvvSTRDP3}).

\end{proof}
From these two lemmas we have the following result.

\begin{thm}\label{thmEQT}
Polyhedrons of two relaxed formulations $\BVV^{tot}$ and $\RR^{tot}$  are equivalent.

%Formulations $BVV^{R}$ and $RR^{R}$ are equivalent.
\end{thm}

\begin{thm}
Polyhedrons of two relaxed formulations $\BVV$ and $\RR$ are equivalent.
\end{thm}

Proof of this theorem is omitted since it is analogous to the proof of Theorem \ref{thmEQT}.

\section{Constraint programming formulation}

 Constraint programming (CP) is a  declarative programming  paradigm wherein relations between variables are defined in the form of constraints.
 Comparing to the imperative programming languages, in  CP the constraints are not organized as commands or set of commands which are sequentially
 executed, but they specify the properties of the solution.

 Detailed description of the CP methods is out of the scope of this paper and can be found in \citep{rossi2006handbook}.

Constraint programming (CP) formulations for S(T)RDP follow directly from previously described problem definitions.
CP formulations use a vector of $n=|V|$ variables denoted by the vector $\mathbf{z}=(z_1,z_2,...,z_n)$.

The correspondence with the definition (\ref{eq:f}) of the function $f$ is
obtained by defining $z_i = f(i)$ for each $i\in V$. From this definition, we derive the objective function (\ref{eq:CPf}) and  the feasible domain for each variable (\ref{eq:CPdom}). The constraint (\ref{eq:CPc1}) is derived from the implication (\ref{eq:c1}). From (\ref{eq:c2srdp}), respectively (\ref{eq:c2strdp})  we derive the relations (\ref{eq:CPc2srdp}) respectively (\ref{eq:CPc2strdp}). Thus, the formulation for SRDP is as follows:

\begin{equation}
\min_{\mathbf{z}} \sum_{i=1}^{n} {z_i}
\label{eq:CPf}
\end{equation}

\begin{equation}
z_i \in \{-1,1,2\}, \quad i \in \{1,...,n\}
\label{eq:CPdom}
\end{equation}

\begin{equation}
(z_i \neq -1) \vee  \bigvee\limits_{j \in N(i)} (z_j = 2), \quad i \in \{1,...,n\}
\label{eq:CPc1}
\end{equation}

\begin{equation}
\sum_{i \in N[i]}{z_i} > 0,\quad i \in \{1,...,n\}.
\label{eq:CPc2srdp}
\end{equation}

Similarly, in the case of STRDP, equations (\ref{eq:CPf}), (\ref{eq:CPdom})and (\ref{eq:CPc1}) are the same, while constraints (\ref{eq:CPc2srdp}) are replaced with:

\begin{equation}
\sum_{i \in N(i)}{z_i} > 0,\quad i \in \{1,...,n\}.
\label{eq:CPc2strdp}
\end{equation}

%Constraints (\ref{eq:CPdom}) define the feasible domain for each variable.
%Constraints (\ref{eq:CPc1}) correspond to constraints (\ref{eq:c1}) from the problem definition, i.e. they enforce that all nodes which are assigned value -1 have at least one neighbor that is assigned value of 2 (notice that implication $X \implies Y$ is here reduced to a form $\neg X \lor Y$).
%Finally, constraints (\ref{eq:CPc2srdp}) and (\ref{eq:CPc2strdp}) correspond to constraints (\ref{eq:c2srdp}) and (\ref{eq:c2strdp}), respectively.

\section{Variable neighborhood search method for SRDP and STRDP}

Variable neighborhood search (VNS) is an optimization metaheuristic, firstly proposed in \citep{mladenovic1997variable}
.

VNS belongs to the class of "single point search" optimization techniques, since it improves a specific solution by systematically exploring its neighborhoods. In order to avoid premature convergence in a local suboptimum, VNS uses a specific strategy of changing the neighborhoods, following the empirical observation that locally optimal solutions found in different neighborhoods are correlated in a way that they also could hold pieces of information about
the global optimum.
As a robust, effective and very adaptable metaheuristic, in recent years VNS has been used  for solving a wide range of optimization problems. The review of these results is out of the scope of this paper and the reader is advised to look in \citep{hansen2010variable} for a detailed overview of VNS methods and its applications.

The simplified scheme of the VNS can be described as follows:

\begin{itemize}
  \item Method is based on the single solution that is iteratively improved;
	\item Process starts with an initial solution that is usually randomly chosen;
  \item Then, iteratively, the following process is done:
	\begin{itemize}
		\item The solution is moved temporarily to a certain neighborhood solution w.r.t. the given neighborhood structure (this process is called shaking);
		\item The temporarily moved solution is then systematically improved by using local search procedure (LS);
		\item Finally, the permanent movement is performed if the new solution is better than previous solution.
	\end{itemize}
\end{itemize}

The most important tasks that must be addressed when dealing with the VNS method are constructing the neighborhood structure (along with shaking mechanism) and local search procedure. Neighborhood structure and shaking mechanism are necessary in providing exploratory, i.e. diversification properties to a method.
On the other hand, the local search procedure is vital for controlling intensification nature of VNS.
LS therefore considers systematic check of nearby solutions w.r.t. solution obtained during the shaking.
Since it is exhaustive, LS needs to be carefully designed and efficiently implemented for a given problem.

The neighborhood structure is a finite list of neighborhoods $N=(N_{k_{min}},...,N_{k_{max}})$,
while $N_k(\mathbf{z})$ is a set of solutions in the $k$-th neighborhood of the solution $\mathbf{z}$.
The most common structure is such that consecutive neighborhoods have increasing cardinality of the solution neighborhoods sets,
i.e. $|N_{k_{min}}(\mathbf{z})|<|N_{k_{min+1}}(\mathbf{z})|<...<|N_{k_{max}}(\mathbf{z})|$.
VNS initially starts with the first neighborhood, thus it sets $k:=k_{min}$. During the shaking process, a candidate solution is randomly selected  from the neighborhood $N_k(\mathbf{z})$ and the local search is then applied to that candidate. If no improvement happen in the local search, the shaking procedure then moves to the  next neighborhood $N_{k+1}$. Otherwise, if the improvement occurs, $k$ is again set to $k_{min}$.

If the improvement does not happen to the point when $k$ reaches $k_{max}$, the $k$ is set back to $k_{min}$.
VNS stops its execution when the stopping criterion is met, e.g. the maximum number of iterations is reached.

\subsection{The overall scheme of VNS for SRDP and STRDP}

Pseudocode of the proposed VNS method for SRDP and STRDP is given in Figure~\ref{figVNSPseudo}. The method uses 4 control parameters: $k_{min}$ and $k_{max}$ are controlling the neighborhood structure, $it_{max}$ defines the maximum number of iterations, while $prob$ parameter is related to the movement procedure. The movement procedure functions as follows: after performing shaking and local search on the current solution $\mathbf{z}$, the newly obtained candidate solution $\mathbf{z''}$ becomes the new current solution if it is strictly better than $\mathbf{z}$. However, in the case when this solution is equally good as the previous, the $prob$ parameter defines the probability threshold for the movement. If a randomly generated number from $[0,1]$ is lower than $prob$, the movement is performed, otherwise the current $\mathbf{z}$ remains unchanged.
The fifth parameter is a flag variable that defines the type of problem, SRDP or STRDP. This parameter is used in the calculation of the objective value, thus it is passed to initialization and local search procedures.

\begin{figure}[h!]
\centering
\includegraphics{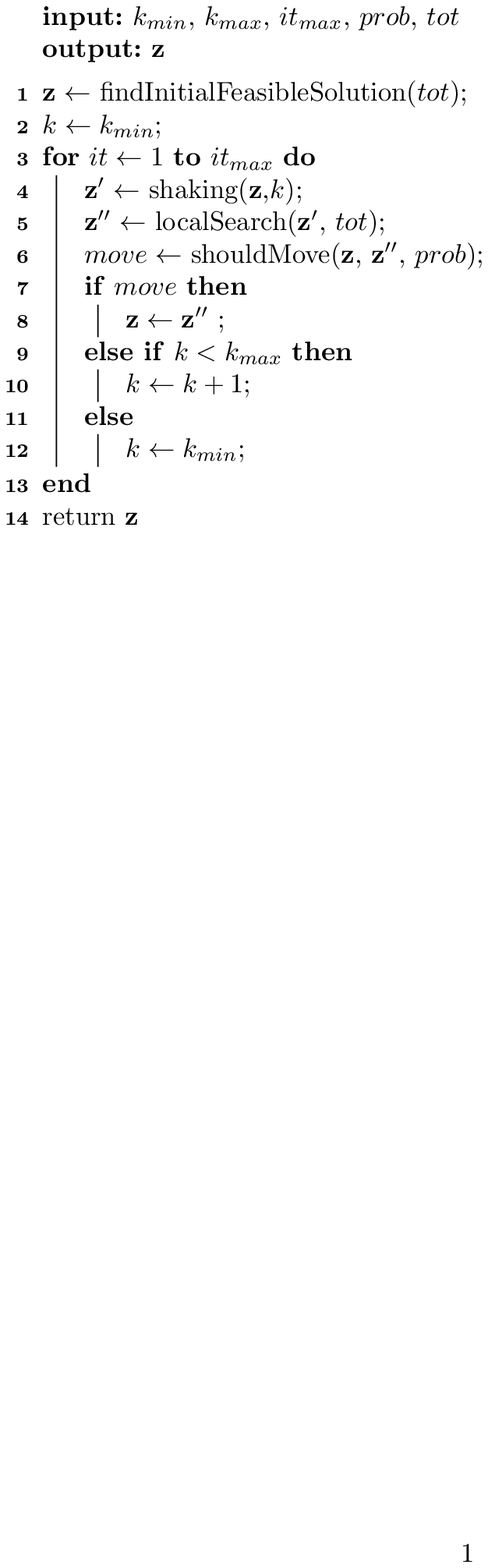}
 \caption{The proposed VNS method for SRDP and STRDP}
 \label{figVNSPseudo}
\end{figure}

\subsection{Solution representation}
Let $G=(V,E)$ be an undirected graph and  $n=|V|$. The solution of VNS is a vector $\mathbf{x}$ of the size $n$. The values of the corresponding $\mathbf{x}$ coordinates are defined as: $z_i = f(i), \; i \in V$, where $f$ is function introduced in (\ref{eq:f}). Vector $\mathbf{z}$ denotes a feasible solution if and only if the corresponding function $f$ satisfy conditions (\ref{eq:c1}) and (\ref{eq:c2strdp}) in the case of STRDP, or conditions (\ref{eq:c1}) and (\ref{eq:c2srdp}) in the case of SRDP.

\subsection{Objective and penalty functions}

The objective function, denoted by $obj_f$, corresponds to the Roman domination number (\ref{eq:w}). Therefore, it is calculated as a sum of the coordinates of the solution vector $\mathbf{z}$ .

\begin{equation}
obj_f(\mathbf{z}) = \sum_{i=1}^{n}{z_i}
\label{eq:obj}
\end{equation}
From the definition (\ref{eq:obj}), one can notice that it does not include the consideration whether a solution is feasible or not. Therefore, an additional mechanism for dealing with infeasible solutions must be introduced. One possibility to handle an infeasible solution is to set its objective value to infinity which could lead to discarding of that solution from the search space.

%This function makes sense only if the corresponding solution vector represents a feasible solution.

However, during the execution, VNS often needs to assess the quality of the intermediate solutions that tend to be \emph{slightly} infeasible. Therefore, VNS uses penalty function, defined in such a way that it can make more precise difference between two solutions, even if one or both of them are infeasible. We denote the penalty function by $pen_f$ and calculate it as follows:

\begin{equation}
pen_f(\mathbf{z}) = (1+f_1(\mathbf{z})) (1+f_2(\mathbf{z})) - 1 + f_3(\mathbf{z})
\label{eq:pen}
\end{equation}

The function $f_1$  calculates the number of violations of the constraints (\ref{eq:c1}), i.e. the number of nodes with the value -1, which do not have a neighbor with value 2.

\begin{equation}
f_1(\mathbf{z}) = |\{f(z_i)=-1 \land ((\nexists j)\;j \in N(i) \land f(z_j)=2)\;| \; i \in \{1,...,n\}\}|
\label{eq:f1}
\end{equation}

The function $f_2$ measures the degree of satisfiability of the constraints (\ref{eq:c2srdp}) for SRDP (respectively (\ref{eq:c2strdp}) for STRDP). The equation~(\ref{eq:f2}) is related to SRDP (the equation for STRDP case is very similar, it only uses open neighborhood instead).

\begin{equation}
f_2(\mathbf{z}) = \sum_{i=1}^{n}{\left(1-\sum_{j \in N[j]}{z_j}\right)}
\label{eq:f2}
\end{equation}

The function $f_3$ represents a scaling of the  function $obj_f$ to interval [0,1]:

\begin{equation}
f_3(\mathbf{z}) = \frac{obj_f(\mathbf{z}) + n}{3 n}.
\label{eq:f3}
\end{equation}

Scaling truly maps to [0,1], since the assignment of the minimal total value corresponds to the scenario where all nodes have the value of -1, which gives a total value of $n$. On the other hand, the assignment of the maximal total value corresponds to the scenario where all nodes receive the value of 2, which in total gives a value of $2n$.

Finally, the overall interpretation of $pen_f$ can be explained. We can notice that the expression $(1+f_1(\mathbf{z})) (1+f_2(\mathbf{z})) - 1$ is a kind of penalty component, that penalizes infeasible solutions w.r.t. the number of unsatisfied  constraints (\ref{eq:c1}) and the total magnitude of unsatisfied constraints (\ref{eq:c2srdp}), respectively (\ref{eq:c2strdp}). If solution is feasible,  the expression $(1+f_1(\mathbf{z})) (1+f_2(\mathbf{z})) - 1$ is equal to zero.
The remaining component $f_3(\mathbf{z})$ is related to the actual quality of the solution w.r.t. the Roman domination number.

\subsection{Intialization}

The initial solution uses the zero vector $\mathbf{z}$ as a starting point. After this initial step, a random value from the set $\{-1,1,2\}$ is iteratively assigned to a randomly chosen node, until the solution  becomes feasible. When this iterative process is finished, the solution is remembered, and an improvement, called \emph{improvement probing} is performed:

\emph{Improvement probing} is done by trying to decrease the nodes with the value 1 to the value -1, or nodes with the value 2 to the value of 1. If this decrement still produces a feasible solution, that new solution is remembered as the current one.

\subsection{Shaking}

 For the given value $k$, the shaking procedure randomly chooses some $k$ (out of $n$) nodes of the solution vector $\mathbf{z}$ and  increases their values in order to obtain the next feasible value. Candidate nodes are only those that can be increased not to violate the constraints, i.e. the nodes with the value of 2 are ignored.
At the same time, the shaking chooses mostly $k$ nodes to decrease their values. In this case, the candidate nodes are those with the values 1 or 2. For each increased node there must be the corresponding decreased node, therefore, the process consists of a loop of mostly $k$ iterations, where in each iteration, a pair of nodes is selected: one suitable for decrement, and the other, suitable for  increment.

%ovde dodati ono dodatno ogranicenje iz nove verzije vns-a
%da ako se izbacuje nesto sto ima vrednost 2, onda ne sme da se
%ubaci nesto sto ima -1, jer bi u tom slucaju ukupna razliak bila pozitivna

If the solution vector remains feasible after these changes, the previously described \emph{improvement probing} is performed.

\subsection{Local search}

Local search systematically tries exclusion of all positively valued nodes for a given solution vector.
Then, for a fixed excluded node, the procedure finds the best node to be included. One node is better to be included than another if it yields lower penalty function value w.r.t. (\ref{eq:pen}).
If some swapping pair (exclusion/inclusion nodes) corresponds to a penalty value better than the current solution, the swap is immediately applied, and LS continues. If all positively valued nodes are tried without success, LS finishes.

Penalty function calculation is efficiently implemented so that it does not require full recalculation, but only the partial calculation that  reuses current penalty value.

The outline of this procedure is shown in Figure~\ref{figLSPseudo}.

\begin{figure}[h!]
\centering
\includegraphics{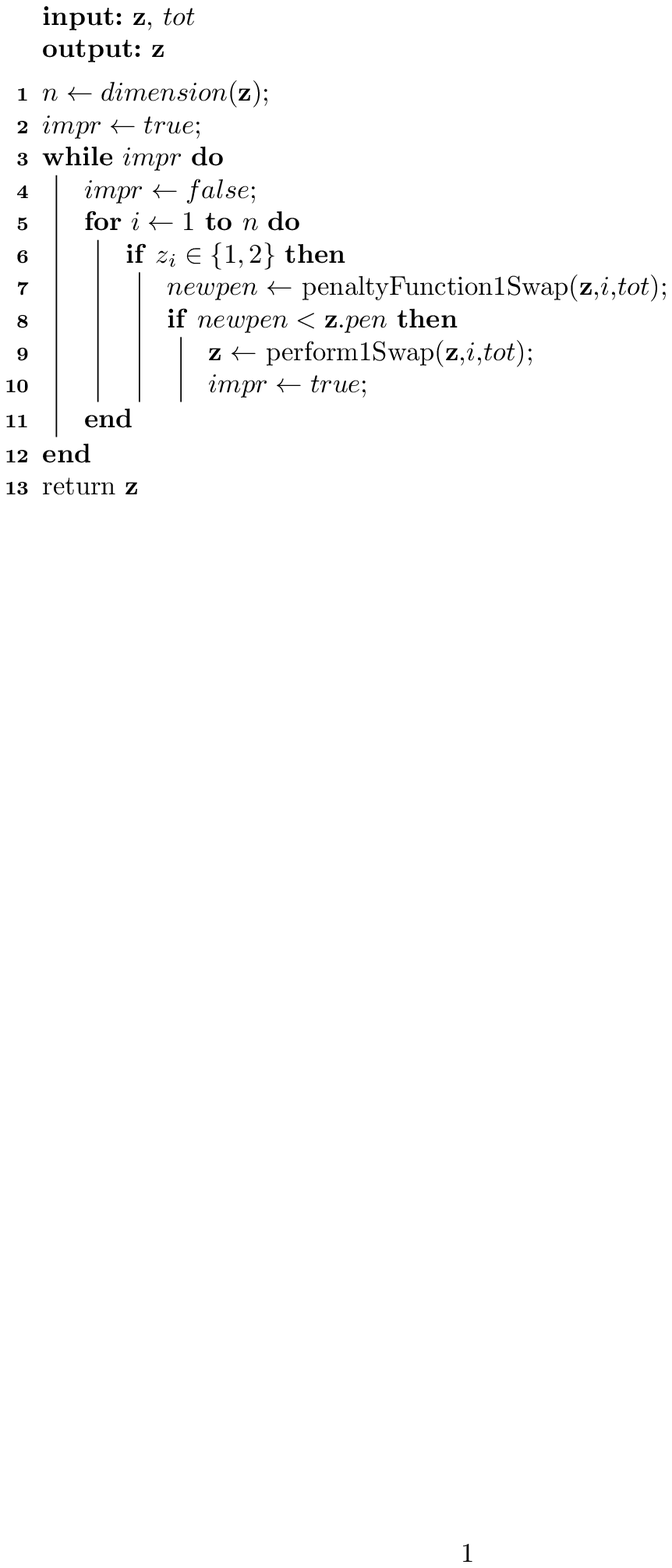}
 \caption{1-swap first improvement local search}
 \label{figLSPseudo}
\end{figure}

\section{Computational results}

This section contains computation results obtained by applying all introduced methods on a large set of instances. All the instances are originally presented in \citep{curro2014roman}. They represent 6 different classes of synthetic graph instances, namely: \emph{bipartite}, \emph{grid}, \emph{net}, \emph{planar}, \emph{random} and \emph{recursive} graphs.

Since all the methods are tested on each set of instances, the complete results of testing are too extensive to be presented in this paper. Therefore, full results are available in the form of online resource\footnote{\url{http://www.matf.bg.ac.rs/~kartelj/sci/strdp/strdp_detailed_results.pdf}} on the website of the third author.

The set of bipartite graphs contains total of 81 instances, with the dimensions from 50 to 400 vertices. The density of generated graphs is controlled by the probability parameter $p$, which varies between 0.01 and 0.9. The greater value of $p$ indicates the greater density of the graph.  The label of a bipartite instance holds information about the dimension of the instance. The first two numbers in the label are the cardinalities of two independent sets, while the third number indicates the density. For example, from the instance's label 100-100-10, one can conclude that the graph contains 200 vertices (two independent sets with 100 vertices in each set) and that the probability of the edge existence between two vertices is 10\%.

The set of grid graphs is the largest one and contains 171 instances. The label of each grid graph indicates its dimension. The smallest grid graph in the set is of the dimension 3x3, while the largest one contains 600 vertices and it is of the dimension 30x20.

Tests are performed on 4 so called net graphs. Each net graph can be constructed from the grid graph by adding additional edges that connect ``the closest'' diagonal vertices. The dimensions of the considered net graphs are from 10x10 to 30x20.

The set of planar graphs contains total of 17 instances. For each planar graph, a pair of geometric coordinates is associated to each vertex, which are further randomly linked by giving a higher probability to nearer vertices. The instance name contain the number of vertices of the planar graph.

The last two classes are random graphs and recursive graphs, containing 72 and 7 instances respectively. The dimensions of random instances are from 50 to 200 vertices, with the given probability of existing of the edges. Similarly to the bipartite instances, the name of an instance contains the number of vertices and the probability of controlling the graph density. The dimensions of recursive graphs are from 7 to 3283 vertices.

Integer programming models $\RR$ and $\BVV$ are executed on 64-bit version of IBM ILOG CPLEX optimization solver, version 12.6.3. We use Python programming language to preprocess the input graphs and to invoke the CPLEX library. The termination of execution can happen in two cases: the optimal solution is found or the time limit of 7200 seconds is reached.

Constraint programming model $CP$ is solved by using 64-bit version of IBM ILOG CP solver, version 12.6.3. Interfacing with CP library is done through C++ programming language. Similarly as CPLEX solver, the CP solver finishes its execution if optimal solution is obtained or if the time limit of 7200 seconds is reached.

The implementation of the VNS is written in C programming language.
Termination criterion is based on the maximum number of iterations $it_{max}=50000$. Other control parameters are set in the following way: $k_{min}=2$, $k_{max}=30$ and $prob=0.5$.

All methods are executed on the  Intel Xeon E5410 @2.33GHz processor, 16GB RAM under the Windows Server 2012 operating system.

In the paper we present the summarized results obtained by all methods.
The results for SRDP are summarized in  Table \ref{tab:tableSummary1}, while the results for STRDP are shown in Table \ref{tab:tableSummary2}. Both tables are organized in the same way. The first column contains the names of the methods. For each method and for each class of graphs, we summarize the following data listed in the second column: total number of achieved optimal solutions, total number of achieved best solutions comparing to the other methods, average value over all instances of each class, as well as the average execution time in seconds over all instances of each class. The following 5 columns contain results obtained by each method on each of 5 classes of instances. In the first row of the table, we put the name of each class of instances as well as the total number of instances in the class. In the last column, we summarize the results obtained on all  instances by each method.

From Table \ref{tab:tableSummary1} one can see that two ILP models are more successful in finding optimal solutions than the other two methods and the $\RR$ model is even more successful than $\BVV$. If we take a look at the last column, we can see that $\RR$ finds 344 out of 352 best solutions, $\BVV$   268 out of 352, while VNS and CP are weaker, obtaining 159 and 197 best solutions. It should be noticed that in many cases, CP could suggest solutions equal to the optimal ones obtained by ILP models, but it could not verify them as optimal in 7200s. Therefore, the average execution time for CP is near to 7200s, since the CP used much time trying to find the results or to verify the optimality of the found results. Because of the specific nature of the recursive instances, in the optimal solutions many vertices are assigned to -1, so the signed (total) Roman domination numbers for all recursive graphs are negative, with the exception of the smallest one. Even without a deeper look at the detailed results, from Table \ref{tab:tableSummary1} that fact can be partially concluded by considering the average values obtained on the recursive instances. Although VNS is less successful than the ILP models, the execution time for this method is less than for the other methods.

\begin{table}
  \centering
  \caption{Summary results for SRDP.}
  \label{tab:tableSummary1}
	\resizebox{\textwidth}{!}{
  \begin{tabular}{|ll|rrrrrr|r|}
	\hline
Method&Measure&\shortstack[r]{Bipartite \\ N=81}&\shortstack[r]{Grid \\ N=171}&\shortstack[r]{Net \\ N=4}&\shortstack[r]{Planar \\ N=17}&\shortstack[r]{Random \\ N=72}&\shortstack[r]{Recursive \\ N=7}&\shortstack[r]{All \\ N=352}\\
\hline
\RR&\#opt&32&133&4&5&40&7&221\\
&\#best&81&167&4&15&70&7&344\\
&avg. value&22.4&28.9&41&1.1&11.3&-323.3&15.6\\
&avg. time&4523.8&1793.4&80&5141.7&3425.8&0.8&2862.2\\
\hline
\BVV&\#opt&31&130&4&5&38&7&215\\
&\#best&41&154&4&6&56&7&268\\
&avg. value&24&29&41&3.4&11.7&-323.3&16.2\\
&avg. time&4675.1&2199.6&113.7&5512.3&3959.4&2.5&3221.8\\
\hline
VNS&\#opt&25&84&3&4&28&4&148\\
&\#best&31&84&3&4&33&4&159\\
&avg. value&24&31.2&41.3&6.2&12.2&-317.6&17.6\\
&avg. time&1504.1&125&1341&5591.6&409.9&7288.9&920.9\\
\hline
CP&\#opt&23&124&4&4&24&4&183\\
&\#best&23&138&4&4&24&4&197\\
&avg. value&36.4&29.6&41&47.4&16.2&-291.3&23\\
&avg. time&7227.9&6688.1&7202.5&6393&7231.6&4130.3&6864.2\\

\hline
\end{tabular}}
\end{table}

From Table \ref{tab:tableSummary2}, we can see the behaviour of the proposed methods applied on solving STRDP. All methods are more successful than in the case of SRDP. $\RR$ model is again more successful than the other methods, finding the best solution in 347 out of 352 cases. The second ILP model is also better than VNS and CP, finding 287 out of 352 best solutions. In the case of STRDP, VNS is more successful than for SRDP and finds  241out of 352 best solutions. CP is also slightly better than in the case of SRDP, finding 209 out of 352 best solutions. Similar to the case of SRDP, the execution time for VNS is less than for the other methods.

$\RR$ model is especially efficient in solving grid graph instances, especially in the case of STRDP, where $\RR$ model succeeds to find almost all optimal solutions (169 out of 171). Both ILP models are also very successful in solving recursive instances, since they succeed to find all optimal solutions for both problems for each recursive instance. The other two methods find 4 out of 7 best solutions for these instances. Since the sets of bipartite and random graphs  contain many large-scale instances, the exact methods cannot find optimal solutions due to time and memory limitations. This is the case with most bipartite and random graph instances containing 100 vertices and more. On the other hand, even for these instances, the ILP models are better than other two methods in most cases. For these large instances, VNS is more successful than CP and finds better results in general.

\begin{table}
  \centering
  \caption{Summary results for STRDP.}
  \label{tab:tableSummary2}
	\resizebox{\textwidth}{!}{
  \begin{tabular}{|ll|rrrrrr|r|}
	\hline
Method&Measure&\shortstack[r]{Bipartite \\ N=81}&\shortstack[r]{Grid \\ N=171}&\shortstack[r]{Net \\ N=4}&\shortstack[r]{Planar \\ N=17}&\shortstack[r]{Random \\ N=72}&\shortstack[r]{Recursive \\ N=7}&\shortstack[r]{All \\ N=352}\\
\hline
RR&\#opt&36&169&1&5&44&7&262\\
&\#best&79&169&4&17&71&7&347\\
&avg. value&25.8&35.6&55.5&0.2&13.2&-281.9&21\\
&avg. time&4200.7&125.5&5407.9&5338&3042.2&2.7&1969.2\\
\hline
BVV&\#opt&34&169&1&5&41&7&257\\
&\#best&43&171&1&7&58&7&287\\
&avg. value&27.2&35.6&57.5&3.1&13.4&-281.9&21.5\\
&avg. time&4451&160&5408.6&5510.1&3738.9&2.9&2194.4\\
\hline
VNS&\#opt&30&164&0&4&30&4&232\\
&\#best&36&164&0&4&33&4&241\\
&avg. value&27&35.7&58.5&5.5&14&-277.9&21.8\\
&avg. time&1863&157&1311.4&7059.6&529.4&9588.8&1159.8\\
\hline
CP&\#opt&25&147&1&4&28&4&209\\
&\#best&25&147&1&4&28&4&209\\
&avg. value&40.3&36.1&64.8&47.2&18.4&-228.9&29.1\\
&avg. time&6976.7&6499.6&7201.2&6386.7&7210.2&5084.5&6729.1\\
\hline
\end{tabular}}
\end{table}

\section{Conclusions}
This paper is devoted to  two variants of the Roman domination problem: the signed Roman domination problem and the signed total Roman domination problem.

For solving each  of two
problems we introduce three exact and one heuristic method: two integer linear programming formulations, constraint programming formulation and the variable neighborhood search heuristic method. For both ILP formulations we provide the proof of the correctness.  Both presented models use polynomial number of variables and constraints, indicating that the models can be used both in theory and practice. By establishing the correspondence between the dominating functions and the variables in the CP model, the CP constraints are directly derived from the constraints of the starting problems.

Variable neighborhood search implements a suitable system of neighborhoods which is based on changing the assigned values to the increasing number of pairs of vertices. In each change, the value assigned to one vertex is increased and the value assigned to another is decreased. Similar system is also applied in the local search procedure, allowing the improvement of the quality of the solution by continuously applying the proposed changes to the pairs of vertices. Specially designed penalty function controls the quality of the solution, allowing the algorithm to handle  the infeasible solutions in order to keep their potential quality.

The experimental results are performed on the large sets of instances from the literature used for solving other Roman domination problems. In the comprehensive experiments, all four methods are tested on each instance, obtaining a huge set of results. The overall conclusion is that two ILP models in most cases are more successful than two other methods.

This research can be extended in several ways. The proposed methods can be applied on solving some other variants of Roman domination problems. It could be also interesting if proposed exact and heuristic approaches are combined in a new hybrid method for solving Roman domination problems.

\section*{Funding}
This work was supported by the Ministry of Education, Science and Technological Development, Republic of Serbia under Grant Number 174010 and  Ministry for Scientific and Technological Development, Higher Education and Information Society, Government of Republic of Srpska, B\&H, under the project ``Development of artificial intelligence methods for solving computational biology problems'' (2020).

\bibliographystyle{tfcse}
\bibliography{mybibfile}

\end{document}